\documentclass[11pt]{amsart}

\usepackage{amssymb,verbatim}
\usepackage{amsmath,amsfonts}
\usepackage[mathscr]{euscript} 
\usepackage{amsthm}
\usepackage{url}
\usepackage{graphicx} 
\usepackage{enumitem} 
\usepackage{hyperref} 
\hypersetup{colorlinks} 
\usepackage{bbm} 
\usepackage{mathrsfs} 
\usepackage{cancel} 

\setlist[enumerate]{
	label=\textnormal{({\roman*})},
	ref={\roman*}}

\makeatletter
\def\th@plain{%
	\thm@notefont{}
	\itshape 
}
\def\th@definition{%
	\thm@notefont{}
	\normalfont 
}
\makeatother

\newtheorem{thm}{Theorem}
\newtheorem{lemma}[thm]{Lemma}

\newtheorem{prop}[thm]{Proposition}

\newtheorem{conv}[thm]{Convention}

\newtheorem{ex}[thm]{Example}

\numberwithin{equation}{section}

\def\sq{\square}

\def\nn{\mathbb N}

\def\ee{\mathbb E}

\def\la{\lambda}
\def\ga{\gamma}

\def\de{\delta}
\def\ep{\ve}
\def\al{\alpha}
\def\be{\beta}

\def\ve{\varepsilon}

\def\cL{\mathcal L}

\def\bP{\mathbf{P}}
\def\ssu{\subset}

\def\<{\langle}
\def\>{\rangle}

\def\Cat{ {\text {\rm Cat} } }
\def\CCat{ {\text {\it Cat} } }

\def\0{{\mathbf 0}}

\def\.{\hskip.06cm}
\def\ts{\hskip.03cm}

\newcommand{\SYT}{\operatorname{{\rm SYT}}}

\def\.{\hskip.06cm}
\def\ts{\hskip.03cm}

\def\nin{\noindent}

\def\Pb{{\text{\bf P}}}

\title[Sorting probability of Catalan posets]{Sorting probability of Catalan posets}
\date{\today}

\author{Swee Hong Chan}
\address[Swee Hong Chan]{Department of Mathematics, UCLA,  Los Angeles, CA 90095.}
\email{\texttt{sweehong@math.ucla.edu}}

\author[\ts Igor Pak]{Igor Pak}
\address[Igor Pak]{Department of Mathematics, UCLA,  Los Angeles, CA 90095.}
\email{\texttt{pak@math.ucla.edu}}

\author[\ts Greta Panova]{Greta Panova}
\address[Greta Panova]{Department of Mathematics, USC,  Los Angeles, CA 90089.}
\email{\texttt{gpanova@usc.edu}}

\begin{document}

\begin{abstract}
We show that the \emph{sorting probability} \ts of the
\emph{Catalan poset}~$P_n$ satisfies \ts $\de(P_n)= O\bigl(n^{-5/4}\bigr)$.
\end{abstract}
	
	
	\maketitle

\section{Introduction}

The \emph{sorting probability} of a poset~$P$, see below,
is an interesting measure of independence of linear extensions of~$P$.
Originally introduced in connection with sorting under partial information
by Kislitsyn (1968) and Fredman (1975), it came to prominence as the subject
of the celebrated \ts \emph{$\frac13\.$--$\,\.\frac23$ \ts Conjecture}, see~\cite{Tro}.
The conjecture received further acclaim in 1980s after a remarkable breakthrough
by Kahn and Saks~\cite{KS}, but remains open in full generality.  We refer
to~\cite[$\S$1.3]{CPP} for a recent overview of the literature and further references.

In this paper we study the sorting probability \ts $\de(P_n)$ \ts of a \emph{Catalan
poset}~$P_n$ on~$2n$ elements, which  is defined as a product  of a chain
with~2 elements and with~$n$ elements: \ts $P_n := C_2 \times C_n$.
The name comes from the fact that the number
of linear extensions of~$P_n$ is the \emph{Catalan number}:
$$
e(P_n) \, = \, \Cat(n) \, := \, \frac{1}{n+1}\binom{2n}{n}.
$$
With over numerous combinatorial interpretations and countless literature,
Catalan numbers are extremely well studied, see e.g.~\cite{Sta,S2}.
It is thus remarkable that \ts $\de(P_n)$ \ts has been out of reach
until now.

\smallskip

Formally, for a finite poset $P=(X,\prec)$, let $\cL_P$ denote the set
of linear extensions of~$P$, and let \ts $e(P):=|\cL_P|$.
The \emph{sorting probability} \ts
$\de(P)$ \ts is defined as
$$
\de(P) \. := \. \min_{x,y\in X} \, \bigl| \ts \Pb\ts[L(x)\leq L(y) ]
\ts - \ts \Pb\ts[L(y)\leq L(x) \ts] \ts\bigr|\.,
$$
where $L \in \cL_P$ is a uniform linear extension of~$P$.
The \ts \emph{$\frac13\.$--$\,\.\frac23$ \ts Conjecture}
mentioned above, claims that \ts $\de(P)\le \frac13$ \ts for all finite posets~$P$.

\smallskip

\begin{thm} \label{t:main}
For the Catalan poset $P_n$, we have \ts $\de(P_n)= O\bigl(n^{-5/4}\bigr)$.
\end{thm}

\smallskip

Until recently, there were very few results in this direction. First,
it was shown by Linial, that $\de(P_n)\le \frac13$\ts, and in fact
this holds for all posets of width two~\cite{Lin}.  For indecomposable posets~$P$ of width two,
this general bound was slightly improved by Sah to \ts $\de(P)< 0.3225$~\cite{Sah}. In an
online discussion about for Catalan posets, the second author improved this bound
to \ts $\de(P_n)< 0.2995$, by comparing $x=(1,17)$, $y=(2,3)$ and taking~$n$
large enough~\cite{Pak}.
In a different direction, Olson and Sagan showed in~\cite{OS}, that \ts
$\de(P_\la)\le \frac13$, for all Young diagrams $\la\vdash n$, s.t.\
\ts $\la\ne (n)$, $(1^n)$.

In our recent paper~\cite{CPP}, we showed
that \ts $\de(P_n)= O\bigl(n^{-1/2}\bigr)$, giving the first bound
that \ts $\de(P_n)\to 0$ \ts as $n\to \infty$. Thus Theorem~\ref{t:main}
is a substantial improvement over this result.  More generally, we showed
that \ts $\de(P_\la) = O\bigl(n^{-1/2}\bigr)$, for all partitions \ts
$\la\vdash n$ \ts with bounded length $\ell =\ell(\la)$,
and such that \ts $\la_\ell= \Omega(n)$.  The tools in~\cite{CPP} rely on
technical results in Algebraic Combinatorics. Here we present a more direct
computation giving better bounds.

\begin{figure}[hbt]
\includegraphics[width=16.3cm]{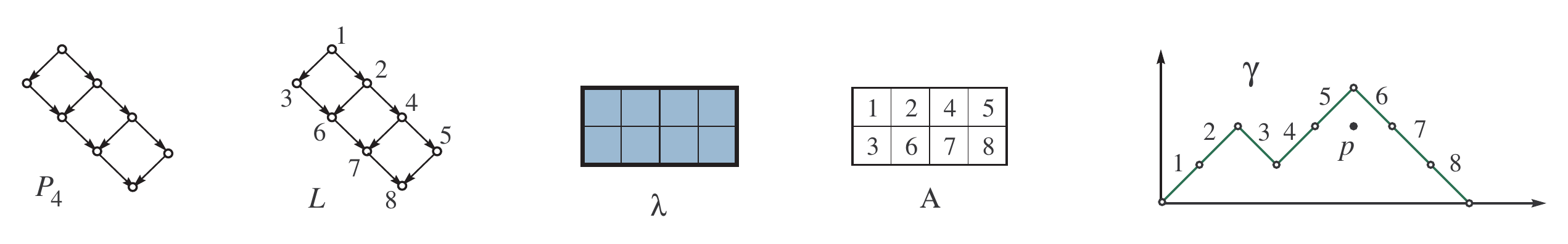}
\vskip-.2cm
\caption{{\small Catalan poset $P_4$, linear extension $L \in \cL_{P_4}$,
Young diagram $\la=(4,4)$, the corresponding standard Young tableau $A\in \SYT(\la)$,
and Dyck path $\ga:(0,0)\to (8,0)$.  }}
\label{f:Cat-path}
\end{figure}

We have several motivations for studying the Catalan posets.  On the one hand,
they are a natural special case of general Young diagram posets $P_\la$,
and Theorem~\ref{t:main} is perhaps an indication how sharp the general
bounds in~\cite{CPP} are. On the other hand, they are closely related to the
behavior of  Brownian excursion, via the standard bijection from standard
Young tableaux $A$ of shape $(n,n)$ to Dyck paths $\ga: (0,0)\to (2n,0)$, see
Figure~\ref{f:Cat-path}.

Curiously, the sorting probability has a natural probabilistic interpretation
in terms of Dyck paths:
$$
\Pb\ts\bigl[\ts L(1,a) < L(2,b)\ts \bigr] \, = \,
\Pb\ts\bigl[\ts \text{$\ga$ passes above $(a+b-1,a-b)$}\ts\bigr]\ts.
$$
For example, for $A\in \SYT(4,4)$ as in the figure, let $a=5$ and $b=2$. Then
we have $A(1,5)<A(2,2)$, and the corresponding path $\ga$ is above point $p=(5,2)$.
Unfortunately the standard probabilistic tools for the Brownian excursion
are too weak to establish Theorem~\ref{t:main}, but they do give the right
heuristic idea of how to approach the problem (see~$\S$\ref{ss:finrem-Brown}).
Thus we resort to a direct asymptotic analysis of the sorting probabilities.

\smallskip

\subsection*{Notation.}
We write \ts $C_1(\ep), C_2(\ep),\ldots$ \ts to denote  (effectively computable)
positive constants that depend on a fixed parameter $\ep>0$, but not on~$n$.
Similarly, we write \ts $C_1, C_2,\ldots$  \ts to denote (effectively computable)
absolute constants that do not depend on~$\ep$.  In the paper, we identify
$P_n$ with Young diagram~$(n,n)$, and linear extensions $\cL_{P_n}$ with
standard Young tableaux $\SYT(n,n)$, see Figure~\ref{f:Cat-path}.
Here we use the matrix coordinates, so e.g.\ $L(2,3)=7$, for $L$
as in the figure.

\bigskip

\section{Sorting probability via lattice paths}	

Throughout the paper, let \. $I=\bigl[\frac{n}{10}, \ts\frac{9n}{10}\bigr]$ \. and \.
$J=\bigl[\frac{\sqrt{n}}{10}, \ts 10\ts\sqrt{n}\bigr]$.
Our approach to proving Theorem~\ref{t:main} is to carefully analyze the
\emph{sorting probability function} \ts $R_n(h,z): I\times J \to [0,1]$,
defined as follows:
\begin{equation}\label{eq:R definition}
 R_n(h,z) \, := \,  \Pb\ts \bigl[ \. L(2,h-z )  \ts < \ts L(1,h) \ts \bigr],
 \ \  \text{where} \ \  h \in I, \ z \in J\ts.
\end{equation}

\smallskip

Consider the \emph{lattice paths}~$\ga$ in $\nn^2$ from \ts $(0,0)$ to $(n,n)$, which
move up and to the right and do not go below (Southeast) of the main diagonal $(0,0)--(n,n)$. Denote by \ts $\CCat(n)$ \ts the set of such paths.
Our intuition comes from the following combinatorial interpretation already
mentioned in the introduction.

\begin{prop}\label{p:Dyck_path}
The sorting probability function \ts $R_n(h,z)$ \ts is equal to the
probability of a lattice path \ts $\ga \in \CCat(n)$ \ts to pass Southeast~\text{\rm (SE)}
of the point \ts $(h-z-\frac12,h-\frac12)$.
\end{prop}

\begin{proof}
This follows from the bijection between lattice paths \ts $\ga \in \CCat(n)$ \ts and
linear extensions \ts $L \in \cL_{P_n}$ \ts via standard Young tableaux \ts $A\in \SYT(n,n)$,
as shown in Figure~\ref{f:Cat-path}.  Formally, let up-steps correspond to
a square in the first row, and  right-steps to  squares in the second row.
The details are straightforward.
\end{proof}

\smallskip


The next two lemmas describe the local behavior of the sorting probability function
\ts $R_n(h,z)$, in essence estimating discrete partial derivatives in both directions.
These lemmas are key to the proof of  Theorem~\ref{t:main}. We prove
the theorem in Section~\ref{s:theorem} and the lemmas in
Section~\ref{s:R-properties}.

\smallskip

To simplify the notation, we extend this function to all real numbers: \ts
$R_n(h,z) := R_n\bigl(\lfloor h\rfloor,\lfloor z\rfloor\bigr)$.

\smallskip

\begin{lemma}\label{l:R is z-increasing function}
For all \. $h \in I$ \. and \. $z \in J$,
the sorting probability function satisfies:
\begin{equation}\label{eq:R-14}
R_n \left(h,\sqrt{n}/10 \right) \, \leq  \, \frac{1}{4}\,
\quad \text{and} \quad R_n(h,10\sqrt{n}) \, \geq \, \frac{3}{4}\.,
\end{equation}
and
\begin{equation}\label{eq:R z derivative}
 \frac{C_1}{\sqrt{n}} \, \le \,  R_n(h,z+1) - R_n(h,z) \, \le  \, \frac{C_2}{\sqrt{n}}\.,
\end{equation}
where \ts $C_1, \ts C_2>0$ \ts are universal constants, and $n$ is large enough.
\end{lemma}

\smallskip

In other words,  the function  \ts $R_n(h,\cdot)$ \ts is increasing and
passing over \ts $1/2$ \ts at some point in the interval~$J$.

\smallskip

\begin{lemma}\label{l:R is h-unimodal}
For all   \. $h \in \bigl[\frac{n}{10},  \frac{n+z+1}{2}\bigr]\ssu I$
\. and \. $z \in J$, the sorting probability function satisfies:
\begin{equation}\label{eq:R-sym}
R_n(h,z) \. = \. R_n(n+z-h,z)\., 
\end{equation}
and
\begin{equation}\label{eq:R h-derivative}
C_3\.\frac{n-2h+z}{n^2} \, \le  \, R_n(h,z) \. - \. R_n(h+1,z) \, \le  \,  C_4\.\frac{n-2h+z}{n^2}\.,
\end{equation}
where \ts $C_3, \ts C_4>0$ \ts are universal constants, and $n$ is large enough.
\end{lemma}

In other words, the function \ts $R_n(\cdot,z)$ \ts is
symmetric, bimodal, and attains its minimum value at \ts
$h = \lfloor \frac{n+z}{2}\rfloor$. See Figure~\ref{f:R-function}
and~\ref{fig:2-graphs} for an illustration.

\begin{figure}[hbt]
\vskip-.3cm
\hskip-2.4cm
\includegraphics[width=13.2cm]{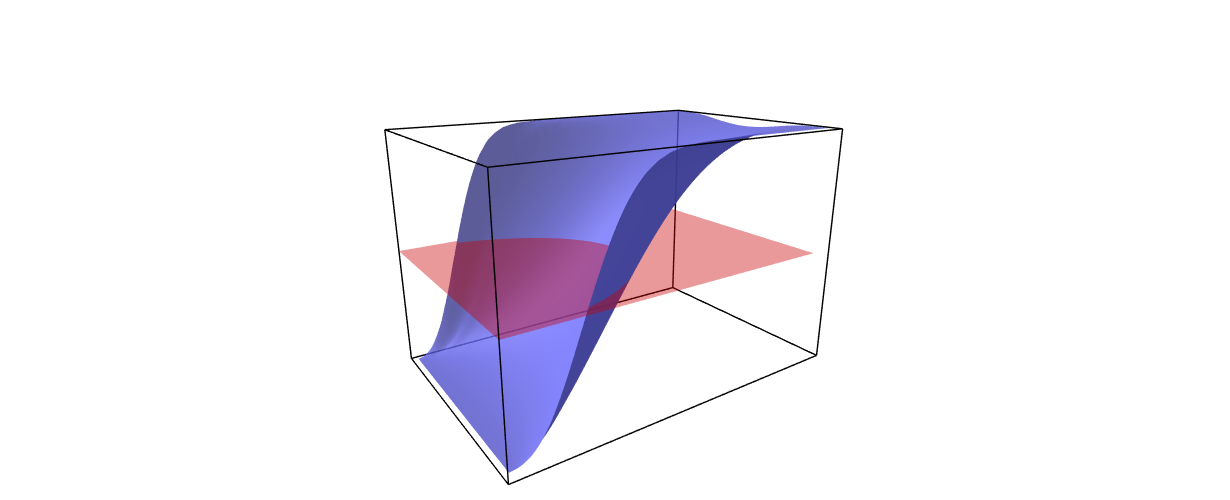}
\hskip-2.1cm
\includegraphics[width=5.6cm]{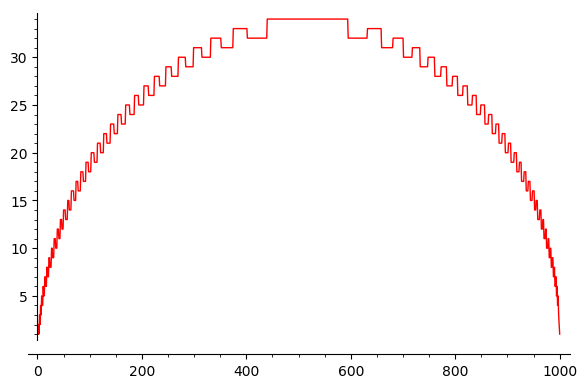}
\vskip-0.2cm
\caption{\small{\underline{Left}: Graph of the function $S(r,t)$ defined in~$\S$\ref{ss:finrem-Brown},
which coincides with the limit surface of the sorting probability~$R_n(h,z)$ when $n\to \infty$.
We also draw the red plane at height \ts $\frac12$ \ts to indicate positions of the best
sorting pairs \ts $x=(1,h)$ \ts and \ts $y=(1,h-z)$. }
\underline{Right}: The intersection between the plane and the surface on the left picture when $n=1000$. The plot is the function $z=f(h)$ that minimizes \ts $\bigl|R_n(h,z)-\frac{1}{2}\bigr|$ \ts for \ts $h \in [0,n]$.}
\label{f:R-function}
\end{figure}

\smallskip

\nin
In fact, the symmetry follows from the central
symmetry of the Catalan poset $P_n$:
$$
\Pb\ts \bigl[ \. L(2,b)  \ts < \ts L(1,a) \ts \bigr] \, = \,
\Pb\ts \bigl[ \. L(2,n-a)  \ts < \ts L(1,n-b) \ts \bigr],
$$
which proves~\eqref{eq:R-sym}.

\bigskip

\section{Proof of Theorem~\ref{t:main}}\label{s:theorem}

By Lemma~\ref{l:R is z-increasing function}, 
there exists  \ts $z \in J$, so that
\begin{equation}\label{eq:not tight}
\frac{1}{2}  \. - \.  \frac{C_2}{\sqrt{n}} \ \leq  \  R_n\bigl(n/2, z\bigr) \ \leq \ \frac{1}{2}\..
\end{equation}

Let \ts $h_0:= n/2 - K\ts n^{3/4}$, where the constant \ts $K>0$ \ts will be determined later.
We have:
\begin{align*}
R_n(h_0,z) \ & = \ R_n\bigl(n/2, z\bigr) \. + \. \sum_{k=0}^{\lfloor K\ts n^{3/4}\rfloor} \.  R_n(h_0+k,z) \. - \. R_n(h_0+k+1,z) \\
& \geq_{\eqref{eq:R h-derivative}} \ R_n\bigl(n/2, z\bigr)  \, + \,
C_3 \sum_{k=0}^{\lfloor K\ts n^{3/4}\rfloor} \. \frac{|k-z|}{n^2} \\
& \geq \quad \ R_n\bigl(n/2, z\bigr)   \, + \,  C_3 \frac{\bigl(K\. n^{3/4}\bigr)^2}{4\.n^2}   \\
& \geq_{\eqref{eq:not tight}} \ \frac{1}{2} \, - \, \frac{C_2}{\sqrt{n}} \, + \, \frac{C_3\ts K^2}{4\sqrt{n}}\..
\end{align*}
Taking \ts $K:=2\sqrt{\frac{C_2}{C_3}}$\ts,
we get
 \begin{equation}\label{eq:h_0 bound}
 R_n(h_0,z) \ \geq \ \frac{1}{2}.
 \end{equation}
It then follows from~\eqref{eq:not tight} and  \eqref{eq:h_0 bound}, that
\[
R_n \bigl(n/2,z\bigr) \, \leq \, \frac{1}{2} \,  \leq \, R_n \bigl(h_0,z\bigr).
\]
Hence, there exists an integer \ts $h_1 \in \left[h_0, n/2\right]$,
such that \ts $R_n(h_1+1,z) \leq \frac12 \leq R_n(h_1,z)$.  We conclude:
 \begin{align*}
\frac{1}{2} \, - \, R_n(h_1+1,z) \  \leq \  R_n(h_1,z) \. - \. R_n(h_1+1,z)  \
\leq_{\eqref{eq:R h-derivative}} \  C_4 \. \frac{2 Kn^{3/4} \ts + \ts 10\sqrt{n}}{n^2}
\ = \ O\bigl(n^{-5/4}\bigr).
 \end{align*}
This completes the proof of the theorem. \ $\sq$

\smallskip

\begin{ex}{\rm The construction in the proof is quite delicate, as it is
fundamentally discrete rather than continuous.
In Figure~\ref{f:proof-ex}, we show the graph of
\ts $R_n(h,z)$ \ts with \ts $n=1000$ and \ts
two values: \ts $z=33$ \ts and \ts $z=34$.
In the former case, the function intersects~$\ts\frac12$,
and \ts $h_1=439$ \ts as in the proof.  In the latter case,
the function is always above~$\ts\frac12$.
}\end{ex}

\begin{figure}[hbt]
	\centering
	\includegraphics[height=4.cm]{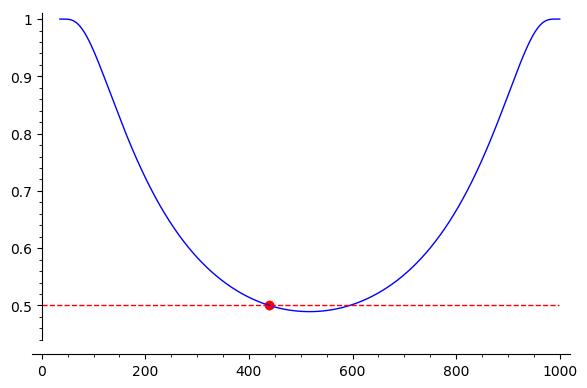}\hskip2.4cm
	\includegraphics[height=4.cm]{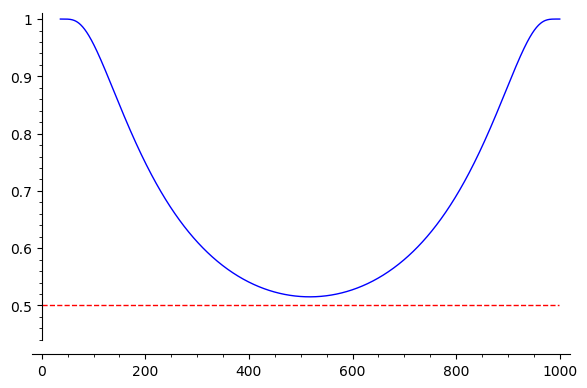}
\caption{\small{Functions \ts $R_{1000}(h,33)$ \ts and \ts $R_{1000}(h,34)$.}}
	\label{f:proof-ex}
\end{figure}

\bigskip

\section{Proof of lemmas}\label{s:R-properties}

\subsection{Preliminaries}
For \ts $0 \leq a \leq b$, denote by \ts $f(a,b)$ \ts
the number of paths \ts $\ga: (0,0) \to (a,b)$ \ts above diagonal $y=x$.

\smallskip
\begin{lemma}[{\rm The ballot theorem, see e.g.\ \cite[$\S$III.1]{Fel}}]\label{l:HLF}
	For $0 \leq a \leq b$,
	\[f(a,b) \ = \  \binom{a+b}{a} - \binom{a+b}{a-1} \ = \ \binom{a+b}{a} \frac{ b-a+1}{b+1}.\]	
\end{lemma}
\smallskip

Let \ts $p_n(a,b)$ \ts be the probability that the lattice path \ts
$\ga \in \CCat(n)$ \ts passes through the point $(a,b)$.  By definition,
$$
p_n(a,b) \, = \, \frac{f(a,b) \cdot f(n-b,n-a)}{\Cat(n)}\..
$$

\smallskip

\begin{lemma}[{\cite[Thm~3.3]{MP14}}]\label{l:visit probability Dyck path}
Fix \ts $\ep, K >0$.  We have:
\begin{equation}\label{eq:visit vertex 1}
p_n(h-z,h) \, \leq \,  \frac{ C_1(\ep) \cdot (z+1)^2}{n^{3/2}} \, e^{-z^2/n}\ts,
\end{equation}
for \. $\ep \ts n \ts \leq \ts h \ts \leq \ts (1-\ep) \ts n$, and \ts $0 \leq z < h$,
and \ts $C_1(\ep)$ \ts a constant independent of~$n$.
Furthermore,
\begin{equation}\label{eq:visit vertex 2}
C_2(\ep,K) \. \frac{(z+1)^2}{n^{3/2}} \, \le \,   p_n(h-z,h)  \, \le \,  C_3(\ep,K) \. \frac{(z+1)^2}{n^{3/2}} \quad \text{for} \quad  1 \. \leq z \. \leq \. K \ts \sqrt{n}\., 		
\end{equation}
where \ts $h,\ts z$ \ts as above, and \ts $C_2(\ep,K)$, $C_3(\ep,K)>0$ \ts are constants independent of~$n$.
\end{lemma}

\smallskip

In fact, when $n\to \infty$, the constants in the theorem are computed explicitly in~\cite{MP14},
but only the upper and lower bounds are needed in the proof of Lemmas~\ref{l:R is z-increasing function}
and~\ref{l:R is h-unimodal}.

\smallskip

Let  \ts $q_n(a,b)$ \ts denote
the probability that the lattice path \ts $\ga \in \CCat(n)$ \ts passes through both \ts
$(a,b-1)$ and $(a,b)$.  Similarly, let \ts $r_n(a,b)$ \ts denote
the probability that the lattice path \ts $\ga \in \CCat(n)$ \ts passes through points \ts
$(a,b-1)$, \ts $(a,b)$ \ts and \ts $(a,b+1)$.  From Lemma~\ref{l:visit probability Dyck path},
we immediately have:
\begin{equation}\label{eq:visit edge 1}
r_n(h-z,h)\. \le \. q_n(h-z,h) \. \leq \. p_n(h-z,h)  \. \le \.
\frac{ C_1(\ep) \cdot (z+1)^2}{n^{3/2}} \, e^{-z^2/n}\ts\.
\end{equation}
This immediately gives the upper bound in
\begin{equation}\label{eq:visit edge 2}
 C_4(\ep,K) \. \frac{(z+1)^2}{n^{3/2}}  \. \le \. r_n(h-z,h) \. \le \.
 q_n(h-z,h) \. \leq \.  C_3(\ep,K) \. \frac{(z+1)^2}{n^{3/2}}
\quad \text{for}  \ \    1 \. \leq z \. \leq \. K \ts \sqrt{n}\., 		
\end{equation}
The lower bound in~\eqref{eq:visit edge 2} follows from
$$
r_n(a,b) = \frac{f(a,b-1)\cdot f(n-b-1,n-a)}{\Cat(n)} \, = \,
p_n(a,b)\. \frac{(n-b)(b-a)(b-a+2)(b+1)}{(2n-a-b)(b-a+1)^2(a+b)}\..
$$
Indeed, for $b=h$ and $a=h-z = b-o(b)$, one can take \ts
$C_4(\ep,K) = C_2(\ep,K)/5$ \ts for \ts $h>\ep n$ large enough.

\medskip

\subsection{Proof of Lemma~\ref{l:R is z-increasing function}}
%
By Proposition~\ref{p:Dyck_path}, the sorting probability function \ts $R_n(h,z)$ \ts
is the probability that the vertical step at height $h$ of a random lattice path \ts
$\ga\in \CCat(n)$ \ts happens at \ts $x \geq h-z$. This gives:
\begin{equation}\label{eq:sum q}
 R_n(h,z) \, = \, \sum_{k=1}^{z} \. q_n(h-k,h)\ts.
 \end{equation}
Since \ts $q_n(h-k,h) \geq 0$, it then follows that  \ts $R_n(h,\cdot)$ \ts
is an increasing function for every~$h$.

\smallskip

Now set \ts $\ep=\frac{1}{10}$, $K=10$, and let \ts $\ep\ts n \leq h \leq (1-\ep)\ts n$.  We have:
\begin{align*}
R_n \bigl(h,\sqrt{n}/10 \bigr) \ & = \
\sum_{k=1}^{\sqrt{n}/10} \. q_n(h-k,h)  \
\leq_{\eqref{eq:visit edge 2}}
\ \sum_{k=1}^{\sqrt{n}/10} \. C_3(\ep,K)\frac{(k+1)^2}{n^{3/2}} \\  &\le \,
C_3(\ep,K) \. \frac{\bigl(\sqrt{n}/10\bigr)^3}{n^{3/2}}
\ = \ \frac{C_3(\ep,K)}{1000}\..
\end{align*}
A direct computer calculation shows that
$$
\frac{C_3\bigl(\frac1{10},10\bigr)}{1000} \.< \. \frac{1}{4}\..
$$
This proves the first inequality in~\eqref{eq:R-14}.

\smallskip

On the other hand, we have:
\begin{align*}
   R_n \left(h,{K\sqrt{n}} \right) \, & = \ \sum_{k=1}^{{K\sqrt{n}}} \. q_n(h-k,h)  \ = \  1 \. - \. \sum_{k>{K\sqrt{n}}}q_n(h-k,h)   \\
   &   \geq_{\eqref{eq:visit edge 1}} \  1 \.  - \.    C_1(\ep) \. \sum_{k>{K\sqrt{n}}} \. \frac{(z+1)^2}{n^{3/2}} \. e^{-z^2/n}  \\
   &  \gtrsim  \quad \ 1 \. - \. C_1(\ep) \. \int_{K}^ \infty  \.  x^2 \ts e^{-x^2} \. d x\ts.
\end{align*}
A direct computer calculation shows that
 for \ts $\ep=\frac{1}{10}$ \ts and \ts $K=10$, we have:
$$
C_1\bigl(0.1\bigr) \, \int_{10}^\infty  \.  x^2 \ts e^{-x^2} \, dx \, < \, \frac{1}{4}\..
$$
This proves the second inequality in~\eqref{eq:R-14}.
\smallskip

For~\eqref{eq:R z derivative}, let \ts $h\in I$, \ts $z\in J$ \ts be as in the lemma.  We have:
\begin{align*}
R_n(h,z+1) \. - \. R_n(h,z) \ =_{\eqref{eq:sum q}} \ q_n(h-(z+1),h)
\end{align*}
and the bounds now follow from~\eqref{eq:visit vertex 2}.
This completes the proof of the Lemma~\ref{l:R is z-increasing function}. \ $\sq$

\medskip

\subsection{Bimodality}
The following lemma is used in the proof of Lemma~\ref{l:R is h-unimodal}
in the next section.

\smallskip

\begin{lemma}\label{l:unimodality}
	Let \ts $h \in [1,n-1]$\ts and \ts $z\in [1,h-1]$\ts.
Then  \ts $R_n(h,z) > R_n(h+1,z)$ \ts if and only if \ts $h \leq \frac12(n+z)\ts$.
\end{lemma}

\begin{proof}
Let \ts $A=(h-z+1/2, h+1/2)$ \ts and \ts $B=(h-z-1/2, h-1/2)$ \ts be two points in the plane.
By Proposition~\ref{p:Dyck_path}, the sorting probabilities \ts $R_n(h+1,z)$ \ts and \ts
$R_n(h,z)$ are  probabilities that the lattice path \ts $\ga\in \CCat(n)$ \ts
passes to SE of the points~$A$ and~$B$, respectively. Denote by $N_1$ and $N_2$,
respectively, the numbers of these paths.  Then we have:
$$
R_n(h,z) \. - \. R_n(h+1,z) \, = \, \frac{1}{\Cat(n)} \. \bigl(N_2 \. - \. N_1\bigr).
$$

Let \ts $x := h+1- z$.  Denote by $M_1$ and $M_2$ the number of paths \ts $\ga\in \CCat(n)$ \ts which
contain segments \ts $(x-2,h) \to (x,h)$ and  \ts $(x-1,h-1)\to (x-1,h+1)$ \ts,
respectively.  Note that \ts $N_2  - N_1$ \ts is exactly the difference between the
numbers of paths passing below point $A$ but above $B$, and the paths passing
left of~$A$ but right of~$B$.  Thus, $N_2  - N_1 = M_2  - M_1$.
\begin{figure}[hbt]
\includegraphics[height=5.2cm]{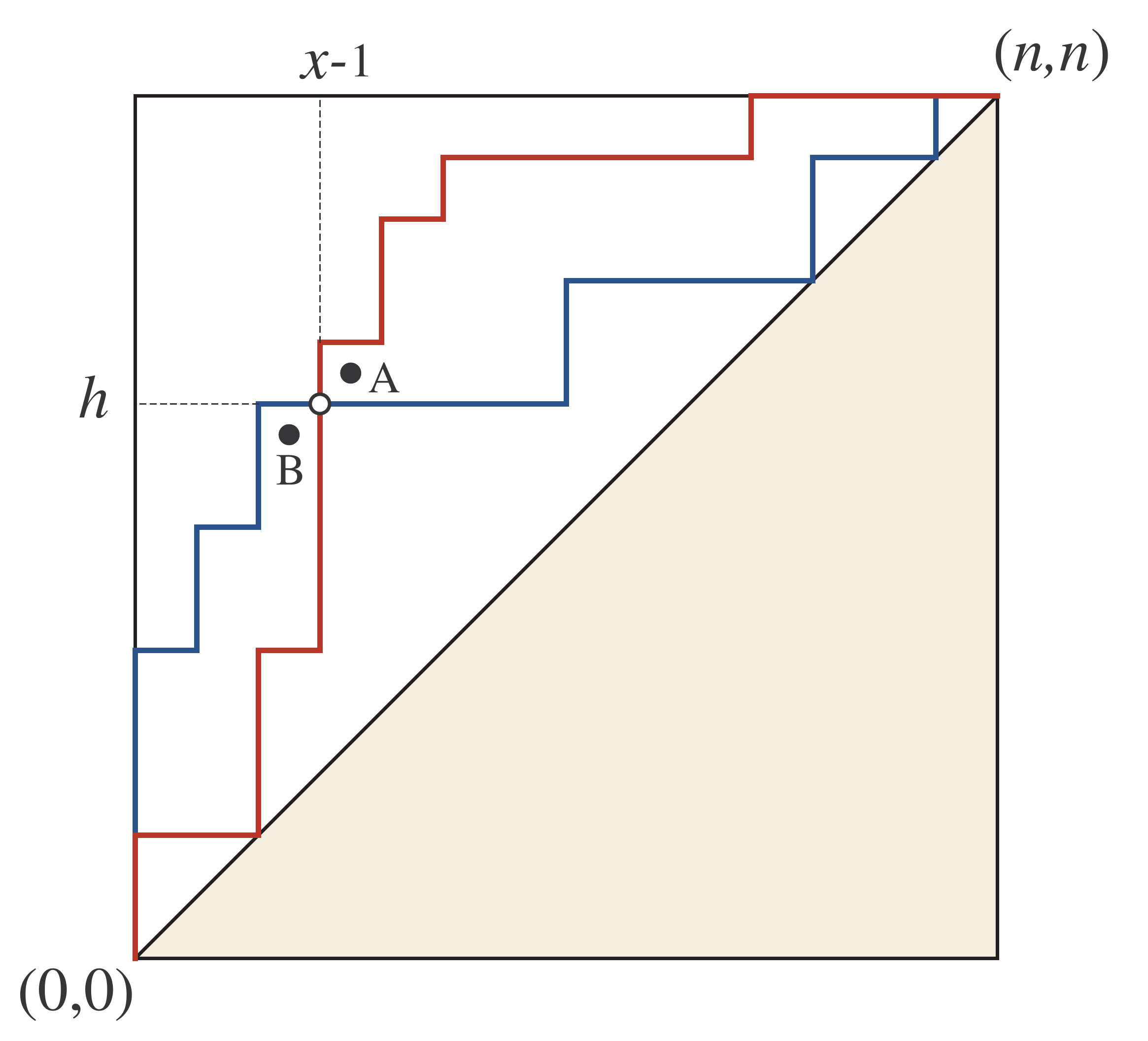}
\caption{The two types of paths in the proof of Lemma~\ref{l:unimodality},
with $M_1$ counting  blue paths and $M_2$  red paths.}
\end{figure}

We have:
{\small
$$\aligned
M_1\, & = \, f(x-2,h) \. f(n-h,n-x) \, = \, \binom{ x+h-2}{x-2} \binom{2n-x-h}{n-h} \frac{ (h-x+3) (h-x+1) }{(h+1)(n-x+1)}\., \\
M_2\, &= \, f(x-1,h-1)\. f(n-h-1, n-x+1)\, = \, \binom{ x+h-2}{x-1} \binom{2n-x-h}{n-h-1} \frac{(h-x+3)(h-x+1)}{h(n-x+2)}
\endaligned
$$
}
and therefore:
{\small
\begin{align*}
M_2 \. - \. M_1 \, = \, \binom{ x+h-2}{x-1} \binom{2n-x-h}{n-h-1} \frac{ (h-x+1)(h-x+2)(h-x+3)(n-x-h+1)}{h (h+1)(n-h)(n-x+2)}.
\end{align*}
}
The last expression is \ts $\geq 0$ \ts if and only if \ts $h+x \leq n+1$, and the result follows.
\end{proof}

\medskip

\subsection{Proof of Lemma~\ref{l:R is h-unimodal} } 	
Equation~\eqref{eq:R-sym} is proved earlier.  For~\eqref{eq:R h-derivative},
from the proof of Lemma~\ref{l:unimodality} we have:
$$
R_n(h,z) \. - \. R_n(h+1,z)  \, = \,  \frac{M_2}{\Cat(n)} \.  \frac{(z+1)(n-2h+z)}{(h+1)(n-h)}
\, = \,  r_n(h-z,z) \.  \frac{(z+1)(n-2h+z)}{(h+1)(n-h)}\..
$$
%
Since \ts $z\in J$, we have:
$$
r_n(h-z,z) \, =_{\eqref{eq:visit vertex 2}} \,  \Theta \left(\frac{1}{\sqrt{n}}\right).
$$
On the other hand, since \ts $h\in I$ \ts  and \ts $z\in J$, we have:
\[
\frac{(z+1)(n-2h+z)}{(h+1)(n-h)}\, = \, \Theta \left(\frac{\sqrt{n}}{n^2}  \. \bigl(n-2h+z\bigr)\right).
\]
Combining these two asymptotics, we conclude:
\[   R_n(h,z)\. - \. R_n(h+1,z)  \, = \, \Theta\left(\frac{n-2h+z}{n^2}\right).
\]
This proves~\eqref{eq:R h-derivative} and completes the proof of
Lemma~\ref{l:R is h-unimodal}. \ $\sq$

\bigskip

\section{Final remarks and open problems}\label{s:finrem}

\subsection{}\label{ss:finrem-Brown}
The sorting probability function \ts $R_n(h,z)$ is the discrete version of the
continuous function
\[
S(t,r)\, :=  \,  \Pb \ts \bigl [ B_0^+(t) \geq r  \bigr]\ts,
\]
where $B_0^+$ is the \emph{Brownian excursion} on $[0,1]$,
defined as the standard Brownian motion conditioned on the event \ts
$B_0^+(0)=B_0^+(1)=0$ and $B_0^+(t) > 0$, for all $t \in (0,1)$.
It  has the following explicit density formula (see e.g., \cite{IM74,Pit}):
\[
S(t,r) \ = \  \frac{2}{\sqrt{2\pi\ts t^3 \ts (1-t^3)}} \ts
\int_{0}^{r} x^2  \ts \exp\left( \frac{-x^2}{2t\ts (1-t)}\right) \, dx\ts,
\]
see Figure~\ref{f:R-function}.
It is shown by Kaigh \cite{Kai76} that \ts $R_n(h,z)$ \ts converges to \ts
$S\left(\frac{h}{n},\frac{z}{\sqrt{2n}} \right)$ \ts as $n \to \infty$.
Unfortunately, the error terms of this convergence are too weak to imply
Theorem~\ref{t:main}.  See Figure~\ref{fig:2-graphs} for a plot comparing
functions \ts $S$ and~$R_{200}$ side by side, and note that these graphs
appear nearly identical on this scale.

\begin{figure}[hbt]
	\centering
	\includegraphics[width=0.32\linewidth]{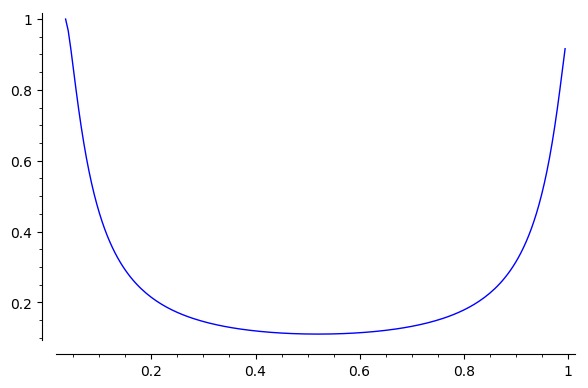} \hskip1.1cm
	\includegraphics[width=0.32\linewidth]{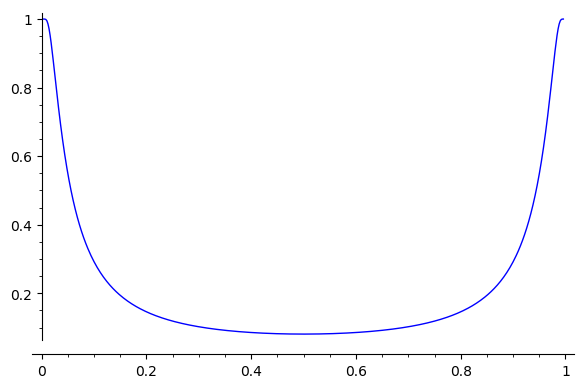}
	\caption{\small{\underline{Left}: \. The plot of  \ts $R_n(tn,\sqrt{n}/2 )$, where \ts $t \in (0,1)$ \ts and $n=200$. \underline{Right}: \. The plot of the probability \ts $1-S(t,r)$, where \ts $t \in (0,1)$ \ts and \ts $r=\sqrt{2}/4$.}}
	\label{fig:2-graphs}
\end{figure}

\subsection{} Lemma~\ref{l:visit probability Dyck path} proved in~\cite{MP14}
is one of  many results in the context of the limit shape of \emph{pattern
avoiding permutations}, see e.g.~\cite{Kit} for an extensive overview of
pattern avoidance.
Many strongly related results are obtained in this direction, too many to
list.  Let us single out papers~\cite{AM,MP16} which are independent of~\cite{MP14},
but cover the same pattern avoidance problem which translates into asymptotics of
Dyck paths. Let us also mention two followup papers~\cite{HRS1,HRS2} which
rederives and extends results in~\cite{MP16,MP14} via Brownian excursions.

\subsection{} In answering the second author's question~\cite{Pak},
Richard Stanley found the following curious limit formulas:
\begin{equation}\label{eq:stan}
\lim_{n\to \infty} \. \ee\bigl[L(1,k-1)\bigr] \, = \, 2k \. - \. \frac{k\binom{2k}{k}}{4^{k-1}} \,, \quad
\lim_{n\to \infty} \. \ee\bigl[L(2,k)\bigr] \, = \, 2k \. + \. \frac{k\binom{2k}{k}}{4^{k-1}} \,,
\end{equation}
where the expectation is over random \ts $L \in \cL(P_n)$.  The limits
for probabilities \ts $\bP\bigl[L(1,a)<(2,b)\bigr]$ \ts for fixed $a>b\ge 1$ also exist,
but much less elegant.  Stanley asked whether there are elegant expectation formulas
similar to~\eqref{eq:stan}, for other partitions~$\la=n\al$.

In principle, using the
technology in~\cite{KS,Saks}, one can use~\eqref{eq:stan} to show that \ts
$\de(P_n)<\frac1e+\ep$ \ts for all $\ep>0$ and $n$ large enough.   \ts
Note that in~\cite{CPP} we already showed that \ts
$\de(P_\la)=O(1/\sqrt{n})$ \ts for the general \emph{TVK case} \ts $\la= n\al$.

\subsection{} It would be interesting to see how tight  Theorem~\ref{t:main} is.
Let
\begin{equation}\label{eq:al-be-def}
\al\, := \, \liminf_{n\to \infty} \. \frac{\log \de(P_n)}{\log n} \quad \text{and}
\quad \be\, := \, \limsup_{n\to \infty} \. \frac{\log \de(P_n)}{\log n}\,.
\end{equation}
We conjecture that
\begin{equation}\label{eq:al-be}
-\infty \, <  \, \al \, < \, \be  \, =  \, -\frac{5}4\,.
\end{equation}
In other words, we believe that our upper
bound is asymptotically tight. On the other hand, we believe that the lower
bound is substantially smaller, but still polynomial.  This has to do with the
fact that $\liminf$ depends on number theoretic properties of $n$ governing
the position of \ts $\frac12$ \ts in the interval \ts $\bigl[R(h+1,z), \ts R(h,z)\bigr]$.
At the moment, we cannot even prove that \ts $\de(P_n)>0$ \ts for all \ts $n\ge 3$.
Finally, most speculatively, we conjecture that
\begin{equation}\label{eq:limsup}
\de(P_n) \. = \. o\bigl(n^{-5/4}\bigr)\ts.
\end{equation}
We refer to~\cite[$\S$12-13]{CPP}
for further discussions and conjectures on the sorting probability.

\subsection{} Our computer calculations show that the sorting probability $\de(P_n)$
has an erratic behavior, but seem to fit well Theorem~\ref{t:main} and the conjectures above.
The first graph in Figure~\ref{f:SP-graph-sage} shows that \ts
$\de(n) \ts n^{5/4}$ \ts is always less than~$3$, but frequently greater than~$1$, and
greater than \ts $\frac13$ \ts at least half the time.  While this may seem to
point against~\eqref{eq:limsup}, we believe it holds since a simple regression
does indicate a very slow trend downwards.

Similarly, the second graph in Figure~\ref{f:SP-graph-sage} shows that
\ts $\log_n \de(n)$ \ts is frequently smaller than \ts $-\frac54$,
but is never too small, suggesting that \ts $-3<\al< \frac32$\.,
in the notation of~\eqref{eq:al-be-def}. Perhaps, going far beyond
\ts $n=1000$ \ts would give further evidence in
support or against the conjectures above.  See the full sequences
\ts $\de(P_n)\ts\Cat(n)$ \ts and \ts $\frac12(1-\de(P_n))\ts\Cat(n)$ \ts
at \cite[\href{https://oeis.org/A335212}{A335212}]{OEIS}
 and \ts \cite[\href{https://oeis.org/A335212}{A335213}]{OEIS},
respectively. 

\begin{figure}[hbt]
\includegraphics[width=7.2cm]{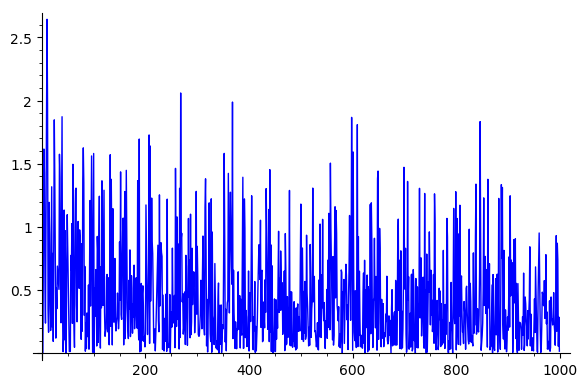} \qquad
\includegraphics[width=7.2cm]{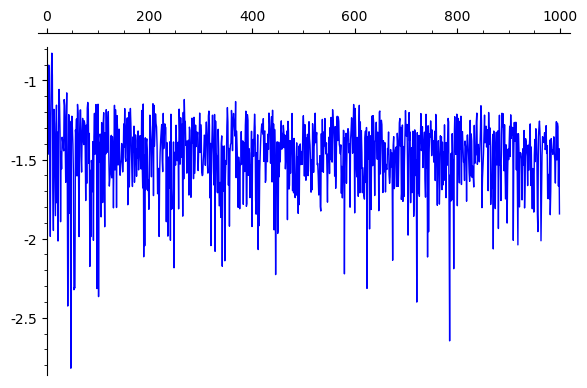}
\vskip-.2cm
\caption{Graphs of \. $\de(P_n)\. n^{5/4}$ \. and \. $\log_n \de(P_n)$,
\ts for \. $3\le n \le 1000$.}
\label{f:SP-graph-sage}
\end{figure}

\subsection{} By the proof of Lemma~\ref{l:unimodality}, the integer \ts
$N_2-N_1\ge 0$ \ts for \ts $h \leq \frac12(n+z)$.  This is a fundamentally
combinatorial statement about the difference in the number of certain lattice paths,
somewhat similar in nature to the \emph{super Catalan numbers}, see e.g.~\cite[$\S$4.5]{P1}
for the references.  It would be interesting to find an explicit combinatorial
interpretation for \ts $(N_2-N_1)$.

\vskip.6cm
	
\subsection*{Acknowledgements}
We are grateful to Sam Hopkins, Han Lyu and Richard Stanley for
interesting discussions and useful comments, and to {\sf MathOverflow}
for providing a convenient platform for such discussions.  The last
two authors were partially supported by the NSF.

\vskip.7cm

\end{document}